\documentclass[11pt,reqno,oneside]{amsart}
\usepackage{graphicx}
\usepackage{amssymb}
\usepackage{epstopdf}

\usepackage[a4paper, total={6in, 9in}]{geometry}

\usepackage{amsmath,amsfonts,amsthm,mathrsfs,amssymb}
\usepackage[usenames]{color}

\usepackage{hyperref}
\usepackage{xcolor}
\hypersetup{
  colorlinks,
  linkcolor={blue!80!black},
  urlcolor={blue!80!black},
  citecolor={blue!80!black}
}

\newtheorem{thm}{Theorem}[section]

\newtheorem{lem}{Lemma}[section]

\theoremstyle{definition}

\theoremstyle{remark}

\numberwithin{equation}{section}

\numberwithin{equation}{section}

\newcounter{saveeqn}

\newcommand{\eqnref}[1]{(\ref {#1})}

\newcommand{\Bn}{\mathbf{n}}

\newcommand{\Bz}{\mathbf{z}}

\newcommand{\Bx}{\mathbf{x}}
\newcommand{\By}{\mathbf{y}}

\newcommand{\Gs}{\sigma}

\newcommand{\Acal}{\mathcal{A}}
\newcommand{\Kcal}{\mathcal{K}}

\newcommand{\Scal}{\mathcal{S}}

\newcommand{\Ocal}{\mathcal{O}}

\newcommand{\ds}{\displaystyle}
\newcommand{\la}{\langle}
\newcommand{\ra}{\rangle}

\newcommand{\RR}{\mathbb{R}}

\newcommand{\p}{\partial}

\newcommand{\beq}{\begin{equation}}
\newcommand{\eeq}{\end{equation}}

\title[Sharp estimates of electric fields in EIT]{Sharp estimate of electric field from a conductive rod and application in EIT}

\author{Xiaoping Fang}
\address{School of Mathematics and Statistics, Key Laboratory of Hunan Province for Statistical Learning and Intelligent Computation, Hunan University of Technology and Business, Changsha 410205, China}
\email{fxp1222@163.com}

\author{Youjun Deng}
\address{School of Mathematics and Statistics, Central South University, Changsha, Hunan, China.}
\email{youjundeng@csu.edu.cn, dengyijun\_001@163.com}

\author{Hongyu Liu}
\address{Department of Mathematics, City University of Hong Kong, Kowloon, Hong Kong SAR, China.}
\email{hongyu.liuip@gmail.com; hongyliu@cityu.edu.hk}

\date{} 
\begin{document}
\maketitle

\begin{abstract}
We are concerned with the quantitative study of the electric field perturbation due to the presence of an inhomogeneous conductive rod embedded in a homogenous conductivity. We sharply quantify the dependence of the perturbed electric field on the geometry of the conductive rod. In particular, we accurately characterise the localisation of the gradient field (i.e. the electric current) near the boundary of the rod where the curvature is sufficiently large. We develop layer-potential techniques in deriving the quantitative estimates and the major difficulty comes from the anisotropic geometry of the rod. The result complements and sharpens several existing studies in the literature. It also generates an interesting application in EIT (electrical impedance tomography) in determining the conductive rod by a single measurement, which is also known as the Calder\'on's inverse inclusion problem in the literature.

\medskip



\end{abstract}

\section{Introduction}

\subsection{Mathematical setup}\label{sect:1.1}
Initially focusing on the mathematics, but not the physics, we consider the following elliptic PDE system in $\RR^2$:
\beq\label{eq:helm01}
\left\{
\begin{split}
\nabla\cdot \Big(\Gs(\Bx) \nabla u (\Bx)\Big) = 0, & \quad \Bx=(x_1,x_2) \in \RR^2,\\
u(\Bx)-H(\Bx)=\Ocal(|\Bx|^{-1}), &\quad |\Bx|\rightarrow \infty,
\end{split}
\right.
\eeq
where $u(\Bx)\in H_{loc}^1(\mathbb{R}^2)$ is a potential field, and $\Gs(\Bx)$ is of the following form:
\beq\label{eq:defGs}
\Gs:= (\Gs_0-1)\chi(D)+ 1,\ \ \Gs_0\in\mathbb{R}_+\ \mbox{and}\ \Gs_0\neq 1,
\eeq
with $D$ being a bounded domain with a connected complement $\mathbb{R}^2\backslash\overline{D}$. $H(\Bx)$ in \eqnref{eq:helm01} is a (nontrivial) harmonic function in $\mathbb{R}^2$, which stands for a background potential. We are mainly concerned with the quantitative properties of the solution $u(\Bx)$ in \eqref{eq:helm01} and in particular, its dependence on the geometry of $D$. To that end, we next introduce the rod-geometry of $D$ for our subsequent study. Let $\Gamma_0$ be a straight line of length $L\in\mathbb{R}_+$ with $\Gamma_0=(x_1, 0)$, $x_1\in (-L/2, L/2)$. Let $\Bn:=(0,1)$ and define the two points $P:=(-L/2, 0)$ and $Q:=(L/2, 0)$. Then the rod $D$ is introduced as $\overline{D}=\overline{D^a}\cup \overline{D^f} \cup \overline{D^b}$, where $D^f$ is defined by
\beq
D^f:=\{\Bx;\ \Bx=\Gamma_0\pm t \Bn, \ t\in (-\delta, \delta)\},\ \ \delta\in\mathbb{R}_+.
\eeq
The two end-caps $D^a$ and $D^b$ are two half disks with radius $\delta$ and centering at $P$ and $Q$, respectively. More precisely,
$$
D^a=\{\Bx;\ |\Bx-P|< \delta, \ x_1< -L/2\}, \quad D^b=\{\Bx;\ |\Bx-Q|< \delta, \ x_1> L/2\}.
$$
It can be verified that $D$ is of class $C^{1,\alpha}$ for a certain $\alpha\in\mathbb{R}_+$. In what follows, we define $S^c:=\p D^c=\p (D^a\cup D^b)$, and $S^f:=\p D^f$. Specially, $S^f=\Gamma_1\cup \Gamma_2$, where $\Gamma_j$, $j=1, 2$ are defined by
\beq
\Gamma_1=\{\Bx;\ \Bx=\Gamma_0-\delta \Bn\}, \quad \Gamma_2=\{\Bx;\ \Bx=\Gamma_0+\delta \Bn\}.
\eeq
Moreover, we shall always use $\Bz_x$ and $\Bz_y$ to signify the projections of $\Bx\in S^f$ and $\By\in S^f$ on $\Gamma_0$, respectively.

The elliptic PDE system \eqref{eq:helm01} describes the perturbation of an electric field $H(\Bx)$ due to the presence of a conductive body $D$. $u(\Bx)$ signifies the electric potential and $\sigma(\Bx)$ signifies the conductivity of the space. The homogeneous background space possesses a conductivity being normalized to be 1, whereas the conductive rod $D$ possesses an inhomogeneous conductivity being $\sigma_0$. The perturbed electric potential is $u-H$ and the gradient field $\nabla(u-H)$ is the corresponding electric current.

\subsection{Background discussion and literature review}

The conductivity equation \eqref{eq:helm01} is a fundamental problem in many existing studies. It is the master equation for the electrical impedance tomography (EIT) which is an important medical imaging modality and it can also find important applications in materials science; see e.g. \cite{HK04, HK07:book, Bor, Uhl} and the references cited therein. There are rich results in the literature devoted to the quantitative properties of the solution to \eqref{eq:helm01} and its geometric relationship to the conductive inclusion $D$. In this work, we shall derive an accurate characterisation of the perturbed electric field $u-H$ and its dependence on the geometry of $D$. There are mainly two motivations for our study as described in what follows.

First, in \cite{AGKP11,ASZZ15,BMVo03,KhZr11}, the authors studied the electric field perturbation from thin/slender structures, which are originated from the study of imaging crack defects in EIT \cite{AGKP11, ASZZ15, Ron05}. This is closely related to the current study. Indeed, the geometric setup in the aforementioned works are more general than the one considered in the present article. However, with the specific rod-geometry, we can derive an accurate characterisation of the perturbed electric field and its geometric dependence on $D$. In fact, in our asymptotic formula of the electric field $u-H$ (with respect to $\delta\ll 1$), the leading-order term is exact and can be used to fully decode $D$. This is in a sharp contrast to the existing studies mentioned above, which inevitably involve some qualitative estimates due to the more general geometries. The rod-geometry, though special, also possesses several local features that are worth our investigation, which is the second motivation of our study as described below.

Clearly, the studies mentioned above are mainly concerned with extracting the global geometry of $\partial D$ from the perturbed electric field $u-H$. On the other hand, there are studies on the relationships between the local geometry of $\partial D$ and the perturbed field $u-H$. In fact, there are classical results concerning the singularities of the solution near a boundary corner point \cite{Gris}. Roughly speaking, if $\partial D$ possesses a corner, then the solution $u$ locally around the corner point can be decomposed into the sum of a regular part and a singular part. Such a qualitative property has been used to establish novel uniqueness and stability results for the Calder\'on inverse inclusion problem in EIT by a single partial boundary measurement \cite{CDL,LB,LT}. The Calder\'on inverse inclusion problem is a longstanding problem in the literature and we shall present more background discussion in Section 4. In \cite{CHL}, the corner singularities of a conductive inclusion have been characterized in terms of the generalised polarisation tensors associated with the electric potential $u$, and the results are directly applied to EIT. Recently, in \cite{LTY}, the authors consider the case that $\partial D$ is smooth, but possesses high-curvature points. In two dimensions, a high-curvature point means that the extrinsic curvature of the boundary curve $\partial D$ at that point is sufficiently large. It is shown in \cite{LTY} that the quantitative property of $\nabla u$ around the high-curvature point enables one to recover the local part of $\partial D$ around that high-curvature point. However, the sharp curvature-dependence of $\nabla u$ in \cite{LTY} is established through numerically refining the upper-bound estimate in \cite{DEF}. As mentioned before, the rod-geometry possesses a few interesting local features that consolidate the numerical study in \cite{LTY}. First, it is geometrically anisotropic where the two dimensions are of different scales. In fact, the curvature at the two end-caps (i.e. $S^c=\partial D^a\cup \partial D^b$) is $\delta^{-1}\geq 1$, whereas the curvature at the facade part (i.e. $S^f$) is $0$. Hence, the rod-geometry, though special, provides rich insights on the curvature dependence of the electric field with respect to the shape of the conductive inclusion. In fact, we shall see that the perturbed electric energy is localized at the two end-caps of the rod. Similar to \cite{LTY}, the result enables us to rigorously justify that one can uniquely determine the conductive rod by a single measurement in EIT. It is emphasized that in three dimensions or in the case that the rod is curved, the situation would become much more complicated. Hence, we mainly consider the case with a straight rod in the two dimensions. Nevertheless, even in such a case, the mathematical analysis is technically involved and highly nontrivial. We shall develop layer potential techniques to tackle the problem. It turns out that the so-called Neumann-Poincar\'e (NP) operator shall play a critical role in our analysis. We would like to mention that the NP operator and its spectral properties have received considerable attentions recently in the literature due to its important applications in several intriguing fields of mathematical physics, including plasmon resonances and invisibility cloaking \cite{ACKL13,ACKL14,AKa16,AKLi16,DLLi18,DLLi182,DLZ20,FDL15,HKLi17,KKLS16,KLY17,LL17}. Finally, we would also like to mention in passing that more general rod-geometries were also considered in the literature in different contexts of physical importance \cite{DLU7,DLU72,LLRU}.

The rest of the paper is organized as follows. In Section 2, we derive several auxiliary results and in Section 3, we present the main results on the quantitative analysis of the solution $u$ to \eqref{eq:helm01} with respect to the geometry of the inclusion $D$. Finally, we consider in Section 4 the application of the quantitative result derived in Section 3 to Calder\'on's inverse inclusion problem.

\section{Some auxiliary results}

In this section, we shall establish several auxiliary results for our subsequent use. We first present some preliminary knowledge on the layer potential operators for solving the conductivity problem \eqref{eq:helm01}, and we also refer to \cite{HK07:book, Ned} for more related results and discussions.

\subsection{Layer potentials}
Let $ G$ be the radiating fundamental solution to the Laplacian $\Delta$ in $\RR^2$, which is given by
\begin{equation}
\label{Gk} \ds G(\Bx) =\frac{1}{2\pi}\ln|\Bx|,
\end{equation}
For any bounded Lipschitz domain $B\subset \RR^2$, we denote by $\Scal_B: L^2(\p B)\rightarrow H^{1}(\RR^2\setminus\p B)$ the single-layer potential operator given by
\beq\label{eq:layperpt1}
\Scal_{B}[\phi](\Bx):=\int_{\p B}G(\Bx-\By)\phi(\By)\; d s_\By,
\eeq
and $\Kcal_B^*: H^{-1/2}(\p B)\rightarrow H^{-1/2}(\p B)$ the Neumann-Poincar\'e (NP) operator:
\beq\label{eq:layperpt2}
\Kcal_{B}^{*}[\phi](\Bx):=\mbox{p.v.}\quad\frac{1}{2\pi}\int_{\p B}\frac{\la \Bx-\By, \nu_\Bx\ra}{|\Bx-\By|^2}\phi(\By)\; d s_\By,
\eeq
where p.v. stands for the Cauchy principle value. In \eqref{eq:layperpt2} and also in what follows, unless otherwise specified, $\nu$ signifies the exterior unit normal vector to the boundary of the domain concerned.
It is known that the single-layer potential operator $\Scal_B$ is continuous across $\p B$ and satisfies the following trace formula
\beq \label{eq:trace}
\frac{\p}{\p\nu}\Scal_B[\phi] \Big|_{\pm} = \Big(\pm \frac{1}{2}I+
\Kcal_{B}^*\Big)[\phi] \quad \mbox{on} \quad \p B, \eeq
where $\frac{\p }{\p \nu}$ stands for the normal derivative and the subscripts $\pm$ indicate the limits from outside and inside of a given inclusion $B$, respectively.

By using the layer-potential techniques, one can readily find the integral solution to \eqnref{eq:helm01} by
\beq\label{eq:intsol01}
u= H + \Scal_D[\varphi],
\eeq
where the density function $\varphi\in H^{-1/2}(\p D)$ is determined by
\beq\label{eq:intsol02}
\varphi= \Big(\lambda I- \Kcal_D^*\Big)^{-1}\Big[\frac{\p H}{\p \nu}\Big|_{\p D}\Big].
\eeq
Here, the constant $\lambda$ is defined by
$$
\lambda:=\frac{\Gs_0+1}{2(\Gs_0-1)}.
$$

\subsection{Asymptotic expansion of the Neumann-Poincar\'e operator}
In what follows, we always suppose that $\delta\ll 1$.
We shall present some asymptotic expansions of the Neumann-Poincar\'e operator with respect to $\delta$.
Recalling that $\p D=S^a \cup \overline{S^f} \cup S^b$, we decompose the Neumann-Poincar\'e operator into several parts accordingly. To that end, we introduce the following
boundary integral operator:
\beq\label{eq:defbnd01}
\Kcal_{\mathcal{S},\mathcal{S}'}[\phi](\Bx):=\chi(\mathcal{S}')\frac{1}{2\pi}\int_{\mathcal{S}} \frac{\la \Bx-\By, \nu_x\ra}{|\Bx-\By|^2}\phi(\By)ds_\By, \quad for\quad \mathcal{S}\cap \mathcal{S}'=\emptyset.
\eeq
It is obvious that $\Kcal_{\Scal,\Scal'}$ is a bounded operator from $L^2(\Scal)$ to $L^2(\Scal')$. For the case $\Scal=S^c$, we mean $S^c=S^a\cup S^b$. In what follows, we define $S_1^a$ and $S_1^b$ by
\beq
S_1^a=\{\Bx;\ |\Bx-P|= 1, \ x_1< -L/2\}, \quad S_1^b=\{\Bx;\ |\Bx-Q|= 1, \ x_1> L/2\}.
\eeq
For the subsequent use, we also introduce the following regions:
\begin{align}\label{}
\iota_{\delta}(P):=\{\Bx;\ |P-\Bz_{x}|=\Ocal(\delta),\ \Bx\in S^f\},\\
\iota_{\delta}(Q):=\{\Bx;\ |Q-\Bz_{x}|=\Ocal(\delta),\ \Bx\in S^f\}.
\end{align}
Define $\tilde\phi(\tilde\Bx):=\phi(\Bx)$, where $\Bx\in S^a, S^b$ and $\tilde\Bx\in S_1^a, S_1^b$.

We can prove the following result on the asymptotic expansion of the Neumann-Poincar\'e operator.

\begin{lem}
The Neumann-Poincar\'e operator $\Kcal_D^*$ admits the following asymptotic expansion:
\beq\label{eq:leasyNP01}
\Kcal_D^*[\phi](\Bx)= \Kcal_0[\phi](\Bx) +\delta\Kcal_1[\phi](\Bx)+ \Ocal(\delta^2),
\eeq
where $\Kcal_0$ is defined by
\beq\label{eq:leasyNP02}
\begin{split}
\Kcal_0[\phi](\Bx)& =\chi(S^a)\Big( \Kcal_{S^f, S^a}[\phi](\Bx)+\frac{1}{4\pi}\int_{S_1^a} \tilde\phi(\By)\Big)+
\chi(S^b)\Big( \Kcal_{S^f, S^b}[\phi](\Bx)+\frac{1}{4\pi}\int_{S_1^b} \tilde\phi\Big) \\
&+\mathcal{A}_{\Gamma_2, \Gamma_1}[\phi]+ \mathcal{A}_{\Gamma_1, \Gamma_2}[\phi]+ \chi(\iota_{\delta}(P))\Kcal_{S^a, S^f}[\phi](\Bx)+\chi(\iota_{\delta}(Q))\Kcal_{S^b, S^f}[\phi](\Bx),
\end{split}
\eeq
and
\beq\label{eq:leasyNP03}
\begin{split}
\Kcal_1[\phi]=& \chi(S^b)\frac{\la \Bx-P, \nu_x\ra}{2\pi|\Bx-P|^2}\int_{S_1^a} \tilde\phi+\chi(S^a)\frac{\la \Bx-Q, \nu_x\ra}{2\pi|\Bx-Q|^2}\int_{S_1^b} \tilde\phi\\
&+ \chi(S^f\setminus\iota_{\delta}(P))\left(\frac{\delta}{|\Bx-P|^2}\int_{S_1^a}(1-\la \tilde\By- P, \nu_x\ra) \tilde\phi(\tilde\By)ds_{\tilde\By}+o\Big(\frac{\delta}{|\Bx-P|^2}\Big)\right)\\
&+\chi(S^f\setminus\iota_{\delta}(Q))\left(\frac{\delta}{|\Bx-Q|^2}\int_{S_1^b}(1-\la \tilde\By- Q, \nu_x\ra) \tilde\phi(\tilde\By)ds_{\tilde\By}+o\Big(\frac{\delta}{|\Bx-P|^2}\Big)\right).
\end{split}
\eeq
Here, the operators $\mathcal{A}_{\Gamma_1, \Gamma_2}$ and $\mathcal{A}_{\Gamma_2, \Gamma_1}$ are defined by
\beq\label{eq:leasyNP0301}
\begin{split}
\mathcal{A}_{\Gamma_1, \Gamma_2}[\phi](\Bx)=&\frac{1}{\pi}\chi(\Gamma_2)\int_{\Gamma_1}\frac{\delta}{|\Bx-\By|^2}\phi(\By)ds_\By, \\
 \mathcal{A}_{\Gamma_2, \Gamma_1}[\phi](\Bx)=&\frac{1}{\pi}\chi(\Gamma_1)\int_{\Gamma_2}\frac{\delta}{|\Bx-\By|^2}\phi(\By)ds_\By.
\end{split}
\eeq
\end{lem}
\begin{proof}
First, one has the following separation:
\beq
\begin{split}
\Kcal_D^*[\phi](\Bx)=&\frac{1}{2\pi}\int_{S^f} \frac{\la \Bx-\By, \nu_x\ra}{|\Bx-\By|^2}\phi(\By)ds_\By+\frac{1}{2\pi}\int_{S^a\cup S^b} \frac{\la \Bx-\By, \nu_x\ra}{|\Bx-\By|^2}\phi(\By)ds_\By\\
=: & A_1[\phi](\Bx)+ A_2[\phi](\Bx).
\end{split}
\eeq
Note that for $\Bx, \By\in \Gamma_j$, $j=2, 3$, one can easily obtain that $\la \Bx-\By, \nu_x\ra=0$. Thus one has
\beq\label{eq:leasy01}
\begin{split}
& A_1[\phi](\Bx)=\frac{1}{2\pi}\int_{S^f} \frac{\la \Bx-\By, \nu_x\ra}{|\Bx-\By|^2}\phi(\By)ds_\By\\
=&\chi(S^a\cup S^b)\frac{1}{2\pi}\int_{S^f} \frac{\la \Bx-\By, \nu_x\ra}{|\Bx-\By|^2}\phi(\By)ds_\By\\
&+
 \chi(\Gamma_1)\frac{1}{2\pi}\int_{\Gamma_2}\frac{\la (\Bx-2\delta\nu_x-\By)+2\delta\nu_x, \nu_x\ra}{|\Bx-\By|^2}\phi(\By)ds_\By\\
 &+\chi(\Gamma_2)\frac{1}{2\pi}\int_{\Gamma_1}\frac{\la (\Bx-2\delta\nu_x-\By)+2\delta\nu_x, \nu_x\ra}{|\Bx-\By|^2}\phi(\By)ds_\By\\
=&\Kcal_{S^f,S^c}[\phi](\Bx) + \delta\chi(\Gamma_1)\frac{1}{\pi}\int_{\Gamma_2}\frac{1}{|\Bx-\By|^2}\phi(\By)ds_\By+\delta\chi(\Gamma_2)\frac{1}{\pi}\int_{\Gamma_1}\frac{1}{|\Bx-\By|^2}\phi(\By)ds_\By.
\end{split}
\eeq
On the other hand,
\beq\label{eq:leasy02}
\begin{split}
&A_2[\phi](\Bx)=\frac{1}{2\pi}\int_{S^a\cup S^b} \frac{\la \Bx-\By, \nu_x\ra}{|\Bx-\By|^2}\phi(\By)ds_\By\\
=&\chi(S^a)\frac{1}{2\pi}\int_{S^a} \frac{\la \Bx-\By, \nu_x\ra}{|\Bx-\By|^2}\phi(\By)ds_\By+
 \chi(S^b)\frac{1}{2\pi}\int_{S^b} \frac{\la \Bx-\By, \nu_x\ra}{|\Bx-\By|^2}\phi(\By)ds_\By\\
 &+\chi(S^a)\frac{1}{2\pi}\int_{S^b}\frac{\la \Bx-\By, \nu_x\ra}{|\Bx-\By|^2}\phi(\By)ds_\By+\chi(S^b)\frac{1}{2\pi}\int_{S^a}\frac{\la \Bx-\By, \nu_x\ra}{|\Bx-\By|^2}\phi(\By)ds_\By\\
 &+\chi(S^f)\frac{1}{2\pi}\int_{S^a\cup S^b}\frac{\la \Bx-\By, \nu_x\ra}{|\Bx-\By|^2}\phi(\By)ds_\By\\
=&\chi(S^a)\frac{1}{4\pi}\int_{S_1^a} \tilde\phi(\tilde\By)ds_{\tilde\By}+\chi(S^b)\frac{1}{4\pi}\int_{S_1^b} \tilde\phi(\tilde\By)ds_{\tilde\By} + \Kcal_{S^b, S^a}[\phi]+\Kcal_{S^a, S^b}[\phi]+\Kcal_{S^c, S^f}[\phi].
\end{split}
\eeq
For $\By\in S^a$ and $\Bx\in S^b$, by using Taylor's expansions one has
\beq\label{eq:leasy03}
|\Bx-\By|=|\Bx-P-(\By-P))|=|\Bx-P-\delta(\tilde\By-P)|=|\Bx-P|+\delta \la \Bx-P, \tilde\By-P\ra +\Ocal(\delta^2).
\eeq
Thus one has
\beq\label{eq:leasy04}
\Kcal_{S^a, S^b}[\phi](\Bx)=\delta\frac{\la \Bx-P, \nu_x\ra}{2\pi|\Bx-P|^2}\int_{S_1^a}\tilde\phi + \Ocal(\delta^2).
\eeq
Similarly, one can obtain
\beq\label{eq:leasy05}
\Kcal_{S^b, S^a}[\phi](\Bx)=\delta\frac{\la \Bx-Q, \nu_x\ra}{2\pi|\Bx-Q|^2}\int_{S_1^b}\tilde\phi + \Ocal(\delta^2).
\eeq
For $\Bx\in S^f$ $\By\in S^a$, by direct computations, one can obtain
$$
\Kcal_{S^a, S^f}[\phi](\Bx)=\delta^2\int_{S_1^a}\frac{1-\la \tilde\By- P, \nu_x\ra}{|\Bx-P|^2-2\delta\la \Bx-P, \tilde\By- P\ra+\delta^2} \tilde\phi(\tilde\By)ds_{\tilde\By}.
$$
We decompose $S^f= (S^f\setminus \iota_{\delta}(P))\cup \iota_{\delta}(P)$, then one has
\beq\label{eq:leasy07}
\Kcal_{S^a, S^f}[\phi](\Bx)=\frac{\delta^2}{|\Bx-P|^2}\int_{S_1^a}(1-\la \tilde\By- P, \nu_x\ra) \tilde\phi(\tilde\By)\,ds_{\tilde\By}+o\Big(\frac{\delta^2}{|\Bx-P|^2}\Big), \quad \Bx\in S^f\setminus\iota_{\delta}(P).
\eeq
Similarly, one can derive the asymptotic expansion for $\Kcal_{S^b, S^f}$.
By substituting \eqnref{eq:leasy04}-\eqnref{eq:leasy07} back into \eqnref{eq:leasy02} and combining \eqnref{eq:leasy01} one finally achieves \eqnref{eq:leasyNP01}, which completes the proof.

The proof is complete.
\end{proof}

\begin{lem}
The operators $\Acal_{\Gamma_j,\Gamma_k}$, $\{j,k\}=\{1, 2\}, \{2, 1\}$, defined in \eqnref{eq:leasyNP0301} are bounded operators from $L^2(\Gamma_j)$ to $L^2(\Gamma_k)$. Furthermore, the operators $\chi(\iota_{\delta}(P))\Kcal_{S^a, S^f}$ and $\chi(\iota_{\delta}(Q))\Kcal_{S^b, S^f}$ are bounded operators from $L^2(S^a)$ to $L^2(S^f)$, and from $L^2(S^b)$ to $L^2(S^f)$, respectively.
\end{lem}
\begin{proof} We only prove that $\mathcal{A}_{\Gamma_2, \Gamma_1}$ is a bounded operator $L^2(\Gamma_2)$ to $L^2(\Gamma_1)$. First, for $\phi_1\in L^2(\Gamma_1)$ and $\phi_2\in L^2(\Gamma_2)$, one has
\[
\begin{split}
&|\la\mathcal{A}_{\Gamma_2, \Gamma_1}[\phi_2],\phi_1\ra_{L^2(\Gamma_1)}|\\
=&\frac{1}{2\pi}\Big|\int_{\Gamma_1}\int_{\Gamma_2}\frac{\delta}{|\Bx-\By|^2}\phi_2(\By)ds_\By\phi_1(\Bx)ds_\Bx\Big| \\
\leq& \frac{1}{4\pi}\int_{\Gamma_1}\int_{\Gamma_2}\frac{\delta}{|\Bx-\By|^2}\phi_2^2(\By)ds_\By ds_\Bx+\frac{1}{4\pi}\int_{\Gamma_1}\int_{\Gamma_2}\frac{\delta}{|\Bx-\By|^2}\phi_1^2(\Bx)ds_\By ds_\Bx \\
=& \frac{1}{4\pi}\int_{-L/2}^{L/2}\int_{-L/2}^{L/2}\frac{\delta}{|x_1-y_1|^2+4\delta^2}dx_1 \phi_2^2(\By) dy_1 + \frac{1}{4\pi}\int_{-L/2}^{L/2}\int_{-L/2}^{L/2}\frac{\delta}{|x_1-y_1|^2+4\delta^2}dy_1 \phi_1^2(\Bx) dx_1\\
=& \frac{1}{8\pi}\int_{-L/2}^{L/2}\Big(\arctan\frac{L/2-y_1}{2\delta}-\arctan\frac{-L/2-y_1}{2\delta}\Big) \phi_2^2(\By) dy_1\\
&+\frac{1}{8\pi}\int_{-L/2}^{L/2}\Big(\arctan\frac{L/2-x_1}{2\delta}-\arctan\frac{-L/2-x_1}{2\delta}\Big) \phi_1^2(\Bx) dx_1\\
\leq& C(\|\phi_1\|_{L^2(\Gamma_1)}^2+\|\phi_2\|_{L^2(\Gamma_2)}^2),
\end{split}
\]
where the constant $C$ is independent of $\delta$. By following a similar arguments as in \cite{HK07:book} (pp. 18), one can show that $\mathcal{A}_{\Gamma_2, \Gamma_1}$ is a bounded operator $L^2(\Gamma_2)$ to $L^2(\Gamma_1)$.
\end{proof}

\begin{lem}\label{le:app01}
Suppose $\Bx\in S^c$, then for any function $\phi\in L^2(S^f)$, which satisfies
\beq
\phi(\By)=-\phi(\By+2\delta\mathbf{n}), \quad \By\in \Gamma_1,
\eeq
there holds
\beq\label{eq:leapp01}
\Kcal_{S^f\setminus(\iota_{\delta}(P)\cup\iota_{\delta}(Q)), S^c}[\phi](\Bx)=o(1).
\eeq
\end{lem}
\begin{proof}
Note that $S^f=\Gamma_1\cup \Gamma_2$. Straightforward computations show that
\[
\begin{split}
&\Kcal_{S^f\setminus(\iota_{\delta}(P)\cup \iota_{\delta}(Q)), S^c}[\phi](\Bx)=\frac{1}{2\pi}\int_{S^f\setminus(\iota_{\delta}(P)\cup \iota_{\delta}(Q))} \frac{\la \Bx-\By, \nu_x\ra}{|\Bx-\By|^2}\phi(\By)ds_\By\\
=&\frac{1}{2\pi}\int_{S^f\setminus(\iota_{\delta}(P)\cup \iota_{\delta}(Q))} \frac{\la \Bx-\Bz_y-\delta\nu_y, \nu_x\ra}{|\Bx-\Bz_y-\delta\nu_y|^2}\phi(\By) ds_\By\\
=&\frac{1}{2\pi}\int_{S^f\setminus(\iota_{\delta}(P)\cup \iota_{\delta}(Q))} \frac{\la \Bx-\Bz_y,\nu_x\ra}{|\Bx-\Bz_y|^2}\phi(\By) ds_\By+o(1)\\
=&\frac{1}{2\pi}\int_{\Gamma_1\setminus(\iota_{\delta}(P)\cup \iota_{\delta}(Q))} \frac{\la \Bx-\Bz_y,\nu_x\ra}{|\Bx-\Bz_y|^2}\phi(\By) ds_\By\\
&+\frac{1}{2\pi}\int_{\Gamma_1\setminus(\iota_{\delta}(P)\cup \iota_{\delta}(Q))} \frac{\la \Bx-\Bz_y,\nu_x\ra}{|\Bx-\Bz_y|^2}\phi(\By+2\delta\mathbf{n}) ds_\By+o(1)=o(1),
\end{split}
\]
which completes the proof.
\end{proof}

\section{Quantitative analysis of the electric field}

In this section, we present the quantitative analysis of the solution to the conductivity equation \eqref{eq:helm01} as well as its geometric relationship to the inclusion $D$.

\subsection{Several auxiliary lemmas}

Recall that $u$ is represented by \eqnref{eq:intsol01}. We first derive some asymptotic properties of the density function $\varphi$ in \eqnref{eq:intsol02}. Let $\Bz_x\in \Gamma_0$ be defined by
\beq
\Bz_x=\left\{
\begin{array}{ll}
\Bx+\delta\mathbf{n}, & \Bx\in \Gamma_1, \\
\Bx-\delta\mathbf{n}, & \Bx\in \Gamma_2.
\end{array}
\right.
\eeq
One has the following asymptotic expansion for $H$ around $\Gamma_0$:
\beq\label{eq:asyH01}
H(\Bx)=H(\Bz_x) + \nabla H(\Bz_x)\cdot (\Bx-\Bz_x)+ \Ocal(|\Bx-\Bz_x|^2)= H(\Bz_x) + \delta \nabla H(\Bz_x)\cdot (\tilde\Bx-\Bz_x)+\Ocal(\delta^2),
\eeq
for $\Bx\in S^f$ and $\tilde\Bx\in S_1^f$. Similarly, one has
\beq\label{eq:asyH02}
H(\Bx)=H(P) + \nabla H(P)\cdot (\Bx-P)+ \Ocal(|\Bx-P|^2)= H(P) + \delta \nabla H(P)\cdot \nu_x+\Ocal(\delta^2),
\eeq
for $\Bx\in S^a$ and $\tilde\Bx\in S_1^a$. Moreover,
\beq\label{eq:asyH03}
H(\Bx)=H(Q) + \nabla H(P)\cdot (\Bx-Q)+ \Ocal(|\Bx-Q|^2)= H(Q) + \delta \nabla H(Q)\cdot \nu_x+\Ocal(\delta^2),
\eeq
for $\Bx\in S^b$ and $\tilde\Bx\in S_1^b$.

We now can show the following asymptotic result.

\begin{lem}\label{le:main01}
Suppose $\varphi$ is defined in \eqnref{eq:intsol02}, then one has
\beq\label{eq:asymvphi01}
\varphi(\Bx)=\left\{
\begin{array}{l}
(\lambda I + A_\delta)^{-1}[(-1)^j\partial_{x_2} H(\cdot,0)](x_1) +\delta(\lambda I -A_\delta)^{-1}[\partial_{x_2}^2 H(\cdot, 0)](x_1)\medskip\\
\quad +\chi(\iota_{\delta^\epsilon}(P)\cup \iota_{\delta^\epsilon}(Q))\Ocal(\delta^{2(1-\epsilon)})+\Ocal(\delta^2), \quad \Bx\in \Gamma_j\setminus(\iota_{\delta}(P)\cup \iota_{\delta}(Q)),\medskip \\
(\lambda I -\Kcal_1^*)^{-1}[\nabla H(P)\cdot\nu ] +o(1), \quad\quad\quad\quad\quad \Bx\in S^a\cup \iota_{\delta}(P), \medskip\\
(\lambda I -\Kcal_2^*)^{-1}[\nabla H(Q)\cdot\nu ] +o(1), \quad\quad\quad\quad\quad \Bx\in S^b\cup \iota_{\delta}(Q),
\end{array}
\right.
\eeq
where $0<\epsilon<1$ and the operator $A_\delta$ is defined by
\beq\label{eq:defAop01}
A_\delta[\psi](x_1):=\frac{1}{2\pi}\int_{-L/2}^{L/2}\frac{\delta}{(x_1-y_1)^2+4\delta^2}\psi(y_1)dy_1, \quad \psi\in L^2(-L/2, L/2).
\eeq
The operators $\Kcal_1^*$ and $\Kcal_2^*$ are defined by
\beq\label{eq:remtm01}
\begin{split}
\Kcal_1^*[\varphi_1](\Bx):=&\int_{S^a\cup \iota_{\delta}(P)}\frac{\la\Bx-\By,\nu_x\ra}{|\Bx-\By|^2}\varphi_1(\By)ds_\By +\chi(\iota_{\delta}(P))\mathcal{A}_{S^f\cap\iota_{\delta}(P)}[\varphi_1](\Bx) \\
\Kcal_2^*[\varphi_2](\Bx):=&\int_{S^b\cup \iota_{\delta}(Q)}\frac{\la\Bx-\By,\nu_x\ra}{|\Bx-\By|^2}\varphi_1(\By)ds_\By+\chi(\iota_{\delta}(Q))\mathcal{A}_{S^f\cap\iota_{\delta}(Q)}[\varphi_2](\Bx),
\end{split}
\eeq
respectively.
\end{lem}
\begin{proof}
Since
$$
\Big(\lambda I- \Kcal_D^*\Big)[\varphi]=\Big[\frac{\p H}{\p \nu}\Big|_{\p D}\Big].
$$
By combining \eqnref{eq:leasyNP01} and \eqnref{eq:asyH01} one can readily verify that
$$
\Big(\lambda I- \Kcal_0\Big)[\varphi](\Bx)= \nabla H(\Bz_x)\cdot\nu_x+ o(1), \quad \Bx\in S^f.
$$
By using \eqnref{eq:leasyNP02}, one thus has
\beq\label{eq:bdint1201}
\begin{split}
\lambda\varphi(\Bx)-\frac{1}{\pi}\int_{\Gamma_2}\frac{\delta}{|\Bx-\By|^2}\varphi(\By)ds_\By=-\nabla H(\Bz_x)\cdot\mathbf{n} +o(1), \quad &\Bx\in \Gamma_1\setminus \big(\iota_{\delta}(P)\cup \iota_{\delta}(Q)\big), \\
\lambda\varphi(\Bx)-\frac{1}{\pi}\int_{\Gamma_1}\frac{\delta}{|\Bx-\By|^2}\varphi(\By)ds_\By=\nabla H(\Bz_x)\cdot\mathbf{n} +o(1), \quad &\Bx\in \Gamma_2\setminus \big(\iota_{\delta}(P)\cup \iota_{\delta}(Q)\big).
\end{split}
\eeq
By direct computations, one can show
\beq\label{eq:bdint1202}
\begin{split}
&\lambda\varphi(x_1,-\delta)-\frac{1}{\pi}\int_{-L/2}^{L/2}\frac{\delta}{(x_1-y_1)^2+4\delta^2}\varphi(y_1,\delta)dy_1\\
&\hspace*{4cm}=-\partial_{x_2} H(x_1,0) +o(1), \quad |x_1|\leq L/2-\Ocal(\delta), \\
& \lambda\varphi(x_1,\delta)-\frac{1}{\pi}\int_{-L/2}^{L/2}\frac{\delta}{(x_1-y_1)^2+4\delta^2}\varphi(y_1,-\delta)dy_1\\
&\hspace*{4cm}=\partial_{x_2} H(x_1,0) +o(1), \quad |x_1|\leq L/2-\Ocal(\delta).
\end{split}
\eeq
Thus one can derive that $\varphi(x_1,-\delta)=-\varphi(x_1,\delta)+o(1)$, for $|x_1|\leq L/2-\Ocal(\delta)$.
Furthermore, for $\Bx\in S_1^a$, by making use of \eqnref{eq:leasyNP02}, \eqnref{eq:asyH02} and Lemma \ref{le:app01}, one has
\beq\label{eq:asymvphi02}
(\lambda I -\Kcal_1^*)[\varphi](\Bx)= \nabla H(P)\cdot\nu_x+ o(1), \quad \mbox{in} \quad S^a\cup\iota_\delta(P).
\eeq
In a similar manner, one can show that
\beq\label{eq:bndintvp02}
(\lambda I -\Kcal_2^*)[\varphi](\Bx)= \nabla H(Q)\cdot\nu_x+ o(1), \quad \mbox{in} \quad S^b\cup\iota_\delta(Q).
\eeq
and so the last equation in \eqnref{eq:asymvphi01} follows.

Next, by combining \eqnref{eq:leasyNP01}, \eqnref{eq:leasyNP02}, and \eqnref{eq:leasyNP03} again for $\Bx\in \Gamma_j\setminus(\iota_{\delta}(P)\cup \iota_{\delta}(Q)) $, $j=1, 2$, and using
the second and third equations in \eqnref{eq:asymvphi01}, one has
\beq\label{eq:bndit0101}
\begin{split}
&\lambda\varphi(x_1,(-1)^j\delta)-\frac{1}{2\pi}\int_{-L/2}^{L/2}\frac{\delta}{(x_1-y_1)^2+4\delta^2}\varphi(y_1,(-1)^{j+1}\delta)dy_1\\
=&(-1)^j\partial_{x_2} H(x_1,0) + \delta \partial_{x_2}^2 H(x_1, 0) +\chi(\iota_{\delta^\epsilon}(P)\cup \iota_{\delta^\epsilon}(Q))\Ocal(\delta^{2(1-\epsilon)})+ \Ocal(\delta^2), \quad 0<\epsilon<1,
\end{split}
\eeq
which verifies the first equation in \eqnref{eq:asymvphi01} and completes the proof.
\end{proof}
Before presenting our main result, we need to further analyze the operator $A_\delta$ defined in \eqnref{eq:defAop01}
\begin{lem}\label{le:main02}
Suppose $A_\delta$ is defined in \eqnref{eq:defAop01}, then it holds that
\beq\label{eq:eigenA01}
A_\delta[y_1^n](x_1)=
\frac{1}{2}x_1^n + o(1),  \quad \Bx\in \Gamma_j\setminus\big(\iota_{\delta}(P)\cup \iota_{\delta}(Q)\big), \quad n \geq 0.\\
\eeq
\end{lem}
\begin{proof}
We use deduction to prove the assertion.
Since $\Bx\in \Gamma_j\setminus\big(\iota_{\delta}(P)\cup \iota_{\delta}(Q)\big)$, one has
$$|L/2-x_1|=\Ocal(\delta^{\epsilon}), \quad \mbox{and} \quad |L/2+x_1|=\Ocal(\delta^{\epsilon}), \quad 0\leq\epsilon<1.$$
Then for $n=0$, it is straightforward to verify that
\beq
\begin{split}
A_\delta[1](x_1)=&\frac{1}{\pi}\int_{-L/2}^{L/2} \frac{\delta}{(x_1-y_1)^2+4\delta^2} dy_1\\
=&\frac{1}{2\pi} \Big(\arctan\frac{L/2-x_1}{2\delta}-\arctan\frac{-L/2-x_1}{2\delta}\Big)=\frac{1}{2}+o(1).
\end{split}
\eeq
Next, we suppose that \eqnref{eq:eigenA01} holds for $n\leq N$. Then by using change of variables, one can derive that
\beq
\begin{split}
A_\delta[y_1^{N+1}](x_1)=&\frac{1}{\pi}\int_{-L/2}^{L/2} \frac{\delta}{(x_1-y_1)^2+4\delta^2} y_1^N y_1dy_1 \\
=&\frac{1}{2\pi}\int_{\frac{-L/2-x_1}{2\delta}}^{\frac{L/2-x_1}{2\delta}} \frac{1}{1+t^2} y_1^N (2\delta t +x_1)dt\\
=&\frac{1}{\pi}\delta\int_{\frac{-L/2-x_1}{2\delta}}^{\frac{L/2-x_1}{2\delta}} \frac{1}{1+t^2} y_1^N t dt + x_1 \frac{1}{2}x_1^N + o(1)\\
=&\frac{1}{\pi}\delta \Ocal\left(\ln(1+\delta^{2(\epsilon-1)})\right)+\frac{1}{2}x_1^{N+1} + o(1)= \frac{1}{2}x_1^{N+1} + o(1),
\end{split}
\eeq
which completes the proof.
\end{proof}
The following lemma is also of critical importance
\begin{lem}\label{le:app01}
There holds the following that
\beq\label{eq:leapp01}
\begin{split}
\int_{S^a\cup \iota_{\delta}(P)}(\lambda I -\Kcal_1^*)^{-1}[\nabla H(P)\cdot\nu ]=&-2\delta\Big(\lambda-\frac{1}{2}\Big)^{-1}\p_{x_1} H(P)+o(\delta), \\
\int_{S^b\cup \iota_{\delta}(Q)}(\lambda I -\Kcal_2^*)^{-1}[\nabla H(Q)\cdot\nu ]=&2\delta\Big(\lambda-\frac{1}{2}\Big)^{-1}\p_{x_1} H(Q)+o(\delta).
\end{split}
\eeq
\end{lem}
\begin{proof}
For any $f\in L^2(\p D)$, we consider the following boundary integral equation
\beq\label{eq:lepfapp01}
(\lambda I -\Kcal_D^*)[\phi]=f.
\eeq
By using the decomposition \eqnref{eq:leasyNP02} (see also \eqnref{eq:asymvphi02} and \eqnref{eq:bndintvp02}), one has
\beq\label{eq:lepfapp02}
\begin{split}
&\chi(S^a\cup \iota_\delta(P))(\lambda I -\Kcal_1^*+o(1))[\phi]+\chi(S^b\cup \iota_\delta(Q))(\lambda I -\Kcal_2^*+o(1))[\phi]\\
&+\mathcal{A}_{\Gamma_2, \Gamma_1}[\phi]+ \mathcal{A}_{\Gamma_1, \Gamma_2}[\phi]+\chi(\iota_{\delta^\epsilon}(P)\cup \iota_{\delta^\epsilon}(Q))\Ocal(\delta^{2(1-\epsilon)}) +\Ocal(\delta^2)=f, \quad 0<\epsilon<1.
\end{split}
\eeq
Note that $\p D$ is of $C^{1, \alpha}$. By taking the boundary integral of both sides of \eqnref{eq:lepfapp01} on $\p D$ and making use of \eqnref{eq:lepfapp02}, one then has
\beq\label{eq:lepfapp03}
\begin{split}
\Big(\lambda-\frac{1}{2}\Big)\int_{\p D}\phi=&\int_{S^a\cup \iota_\delta(P)}(\lambda I -\Kcal_1^*+o(1))[\phi]+\int_{S^b\cup \iota_\delta(Q)}(\lambda I -\Kcal_2^*+o(1))[\phi]\\
&+\int_{\Gamma_1}\mathcal{A}_{\Gamma_2, \Gamma_1}[\phi]+\int_{\Gamma_2}\mathcal{A}_{\Gamma_1, \Gamma_2}[\phi]+o(\delta)=\int_{\p D} f.
\end{split}
\eeq
By assuming $f=\chi(S^a\cup\iota_{\delta}(P))\nabla H(P)\cdot\nu$ and plugging into \eqnref{eq:lepfapp03}, one thus has
\beq
\Big(\lambda-\frac{1}{2}\Big)\int_{S^a\cup \iota_{\delta}(P)}(\lambda I -\Kcal_1^*+o(1))^{-1}[\nabla H(P)\cdot\nu]=\int_{S^a\cup \iota_{\delta}(P)} \nabla H(P)\cdot\nu=-2\delta\p_{x_1} H(P),
\eeq
which verifies the first equation in \eqnref{eq:leapp01}. Similarly, by assuming $f=\chi(S^b\cup\iota_{\delta}(Q))\nabla H(q)\cdot\nu$, one can prove the second equation in \eqnref{eq:leapp01}. The proof is complete.
\end{proof}
\subsection{Sharp asymptotic approximation of the solution $u$}

With Lemmas \ref{le:main01}, \ref{le:main02} and \ref{le:app01}, we can now establish one of the main results of this paper as follows.

\begin{thm}
Let $u$ be the solution to \eqnref{eq:helm01} and \eqnref{eq:defGs}, with $D$ of the rod-shape described in Section~\ref{sect:1.1}. Then for $\Bx\in \RR^2\setminus\overline{D}$, it holds that
\beq\label{eq:thmanin01}
\begin{split}
u(\Bx)=&H(\Bx) + \delta \frac{1}{2\pi}\Big(\lambda-\frac{1}{2}\Big)^{-1}\int_{-L/2}^{L/2}\ln\frac{(x_1-y_1)^2+x_2^2}{(x_1+L/2)^2+x_2^2} \p_{y_2}^2H(y_1,0)\,dy_1\\
& \quad \quad \, + \delta \frac{1}{\pi}\Big(\lambda-\frac{1}{2}\Big)^{-1}\int_{-L/2}^{L/2}\frac{x_2}{(x_1-y_1)^2+x_2^2} \p_{y_2}H(y_1,0)\,dy_1 \\
&  \quad \quad \, + \delta \frac{1}{2\pi}\Big(\lambda-\frac{1}{2}\Big)^{-1}\ln\frac{(x_1-L/2)^2+x_2^2}{(x_1+L/2)^2+x_2^2}\p_{x_1}H(L/2,0)+o(\delta).
\end{split}
\eeq
\end{thm}
\begin{proof}
By using \eqnref{eq:intsol01} and Taylor's expansion along with $\Gamma_0$, one has
\beq\label{eq:thpf01}
\begin{split}
u(\Bx)=&H(\Bx)+ \int_{S^f\setminus(\iota_{\delta}(P)\cup \iota_{\delta}(Q))}G(\Bx-\Bz_y)\varphi(\By)\,ds_\By \\
&\quad \quad \ \,+ \delta \int_{S^f\setminus(\iota_{\delta}(P)\cup \iota_{\delta}(Q))}\nabla_\By G(\Bx-\Bz_y)\cdot \nu_\By \varphi(\By)\,ds_\By\\
&\quad \quad \ \, + \int_{S^a\cup \iota_\delta(P)} G(\Bx-\Bz_y) \varphi(\By)\,ds_{\tilde\By}+ \int_{S^b\cup \iota_\delta(Q)} G(\Bx-\Bz_y) \tilde\varphi(\tilde\By)\,ds_{\tilde\By}+\Ocal(\delta^2).
\end{split}
\eeq
First, by using \eqnref{eq:bndit0101}, one can derive that
\beq\label{eq:thpf02}
\begin{split}
& \int_{S^f}G(\Bx-\Bz_y)\varphi(\By)\,ds_\By\\
=&2\delta \int_{\Gamma_1\setminus(\iota_{\delta}(P)\cup \iota_{\delta}(Q))}G(\Bx-\Bz_y)(\lambda I -A_\delta)^{-1}[\partial_{x_2}^2 H(\cdot, 0)](y_1)dy_1 +o(\delta)\\
=& \delta \frac{1}{2\pi}\Big(\lambda-\frac{1}{2}\Big)^{-1}\int_{-L/2}^{L/2}\ln((x_1-y_1)^2+x_2^2)\p_{y_2}^2H(y_1,0)\,dy_1+o(\delta).
\end{split}
\eeq
Similarly, one has
\beq\label{eq:thpf03}
\begin{split}
& \int_{S^f}\nabla_\By G(\Bx-\Bz_y)\cdot \nu_\By \varphi(\By)\,ds_\By\\
=& \frac{1}{\pi}\Big(\lambda-\frac{1}{2}\Big)^{-1}\int_{-L/2}^{L/2}\frac{x_2}{(x_1-y_1)^2+x_2^2} \p_{y_2}H(y_1,0)\,dy_1+o(1).
\end{split}
\eeq
By using Lemma \ref{le:app01}, one then obtains that
\beq\label{eq:thpf04}
\begin{split}
&\int_{S^b\cup \iota_\delta(Q)} G(\Bx-\Bz_y) \varphi(\By)\,ds_{\By}=\int_{S^b\cup \iota_\delta(Q)} G(\Bx-Q) \varphi(\By)\,ds_{\By}+o(\delta)\\
=&\delta\frac{1}{\pi}\ln|\Bx-Q|\Big(\lambda-\frac{1}{2}\Big)^{-1}\p_{x_1} H(Q)+o(\delta)\\
=&\delta\frac{1}{2\pi}\Big(\lambda-\frac{1}{2}\Big)^{-1}\ln((x_1-L/2)^2+x_2^2)\p_{x_1}H(L/2,0)+o(\delta).
\end{split}
\eeq
Noting that
$$
\int_{\p D}\varphi(\By)ds_\By=\int_{\p D}\Big(\lambda I- \Kcal_D^*\Big)^{-1}\Big[\frac{\p H}{\p \nu}\Big](\By)ds_\By=0,
$$
and by combining \eqnref{eq:thpf02}, one can readily show that
\beq\label{eq:thpf05}
\begin{split}
&\int_{S^a\cup \iota_\delta(P)} G(\Bx-\Bz_y) \varphi(\By)\,ds_{\By}=\int_{S^a\cup \iota_\delta(P)}  G(\Bx-P)\varphi(\By)\,ds_{\By}+o(\delta) \\
=&-\frac{1}{2\pi}\ln|\Bx-P|\Big(\int_{S^b\cup \iota_\delta(Q)} \varphi(\By)\,ds_{\By}\\
&\quad\quad \quad\quad+2\int_{\Gamma_1\setminus(\iota_{\delta}(P)\cup \iota_{\delta}(Q))}(\lambda I -A_\delta)^{-1}[\partial_{x_2}^2 H(\cdot, 0)](y_1)dy_1\Big) +o(\delta) \\
=&-\delta\frac{1}{2\pi}\Big(\lambda-\frac{1}{2}\Big)^{-1}\ln((x_1+L/2)^2+x_2^2)\p_{x_1}H(L/2,0)\\
&-\delta\frac{1}{2\pi}\Big(\lambda-\frac{1}{2}\Big)^{-1}\ln((x_1+L/2)^2+x_2^2)\int_{-L/2}^{L/2}\p_{y_2}^2H(y_1,0)\,dy_1+o(\delta).
\end{split}
\eeq
Finally, by substituting \eqnref{eq:thpf02}-\eqnref{eq:thpf05} into \eqnref{eq:thpf01} one has \eqnref{eq:thmanin01}, which completes the proof.
\end{proof}

We finally can derive the sharp asymptotic expansion of the electric field as follows.
\begin{thm}\label{thm:3.2}
Suppose $H(\Bx)=\mathbf{a}\cdot\Bx$, where $\mathbf{a}=(a_1, a_2)\in\mathbb{R}^2$. Then for $\Bx\in \RR^2\setminus\overline{D}$, the electric field $u$ satisfies
\beq\label{eq:thmain02}
\begin{split}
u(\Bx)=&\mathbf{a}\cdot\Bx+ \delta \frac{1}{\pi}\Big(\lambda-\frac{1}{2}\Big)^{-1}a_2 \left(\arctan\Big(\frac{L/2-x_1}{x_2}\Big)+ \arctan\Big(\frac{L/2+x_1}{x_2}\Big)\right)\\
&+\delta \frac{1}{2\pi}\Big(\lambda-\frac{1}{2}\Big)^{-1}a_1\ln\frac{(x_1-L/2)^2+x_2^2}{(x_1+L/2)^2+x_2^2}+o(\delta).
\end{split}
\eeq
Furthermore, the perturbed gradient field admits the following asymptotic expansion:
\beq\label{eq:thmain03}
\nabla u(\Bx)= \mathbf{a} + \delta \frac{1}{\pi}\Big(\lambda-\frac{1}{2}\Big)^{-1}
\left(
\begin{array}{c}
f_2(\Bx) a_1 -f_1(\Bx) a_2\\
f_1(\Bx) a_1 +f_2(\Bx) a_2
\end{array}
\right)+ o(\delta),
\eeq
where the functions $f_j$, $j=1, 2$ are defined by
\beq\label{eq:thmain04}
\begin{split}
& f_1(\Bx):=\frac{x_2}{(x_1-L/2)^2+x_2^2}-\frac{x_2}{(x_1+L/2)^2+x_2^2}\\
&f_2(\Bx):=\frac{x_1-L/2}{(x_1-L/2)^2+x_2^2}-\frac{x_1+L/2}{(x_1+L/2)^2+x_2^2}.
\end{split}
\eeq
\end{thm}
\begin{proof}
The proof is given by using \eqnref{eq:thmanin01} together with direct computations.
\end{proof}

\subsection{Quantitative analysis and numerical illustrations}

Define the following vector field
\beq
\mathbf{E}^s:=\delta\frac{1}{\pi}\Big(\lambda-\frac{1}{2}\Big)^{-1}
\left(
\begin{array}{c}
f_2(\Bx) a_1 -f_1(\Bx) a_2\\
f_1(\Bx) a_1 +f_2(\Bx) a_2
\end{array}
\right).
\eeq
According to \eqnref{eq:thmain03}, $\mathbf{E}^s$ is the leading order term of the perturbed gradient field. It is noted that the distribution of $|\mathbf{E}^s|$ is
independent of the uniform gradient potential $\mathbf{a}$. In fact, one has
\beq\label{eq:re01}
|\mathbf{E}^s|^2=\delta^2\frac{1}{\pi^2}\Big(\lambda-\frac{1}{2}\Big)^{-2}(a_1^2+a_2^2)(f_1(\Bx)^2+f_2(\Bx)^2).
\eeq
Moreover, further computations show that
\beq\label{eq:re02}
f_1(\Bx)^2+f_2(\Bx)^2=\left(\frac{1}{|\Bx-Q|}-\frac{1}{|\Bx-P|}\right)^2+\frac{2}{|\Bx-P||\Bx-Q|}\left(1-\frac{\la \Bx-P, \Bx-Q\ra}{|\Bx-P||\Bx-Q|}\right).
\eeq
One can thus derive that $|\mathbf{E}^s|$ is maximized near the two caps (high curvature parts) of the inclusion $D$. In fact, near the caps one has
$$
|\Bx-P|=\delta+o(\delta), \quad \mbox{or} \quad |\Bx-Q|=\delta+o(\delta).
$$
By \eqnref{eq:re02} one then has
\beq\label{eq:re03}
f_1(\Bx)^2+f_2(\Bx)^2=\delta^{-2}(1+o(1)),
\eeq
while near the centering parts of the rod,
$$
f_1(\Bx)^2+f_2(\Bx)^2=\Ocal(1).
$$
To better illustrate the result, we next present some numerical solutions with different background fields. The parameters of the rod-shape inclusion are selected as follows:
\beq
\sigma_0=2, \quad L=10, \quad \delta=5*\tan(\pi/36)\approx 0.4374.
\eeq
We choose three different uniform background fields, i.e., $\mathbf{a}=(1,0), (0, 1), (1, 1)$, respectively, and plot the absolute values of the perturbed fields as well as the corresponding gradient fields, which are scaled for better display. It is clearly shown from Figure \ref{fig:perturbed_01} to Figure \ref{fig:perturbed_03} that the gradient fields behave much stronger near the high curvature parts of the inclusion.

\begin{figure}[h]
\begin{center}
{\includegraphics[width=0.4\textwidth]{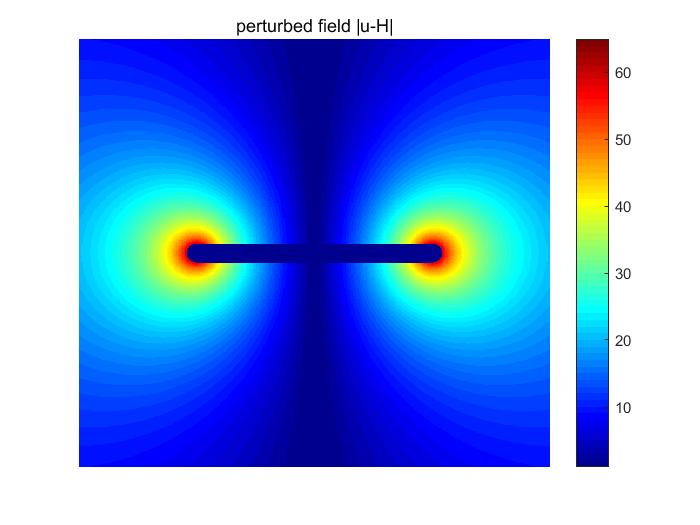}}
{\includegraphics[width=0.4\textwidth]{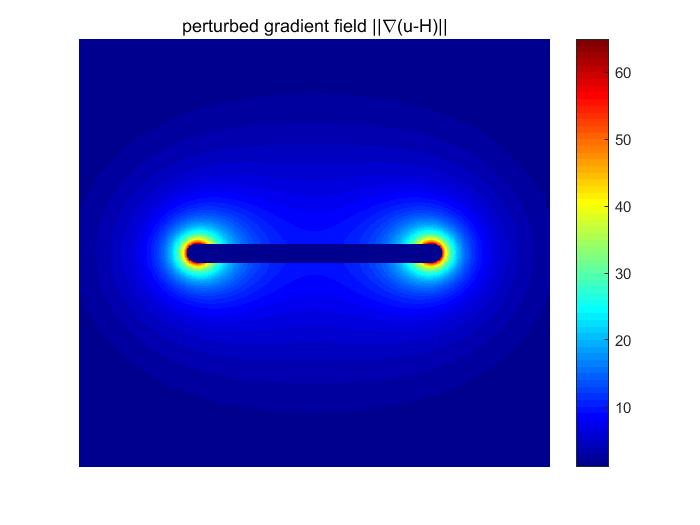}}
\end{center}
\caption{\label{fig:perturbed_01} $\mathbf{a}=(1,0)$. Left: Perturbed field $|u-\mathbf{a}\cdot\Bx|$ (scaled)   Right: Perturbed gradient field $|\nabla u-\mathbf{a}|$ (scaled). }
\end{figure}

\begin{figure}[h]
\begin{center}
{\includegraphics[width=0.4\textwidth]{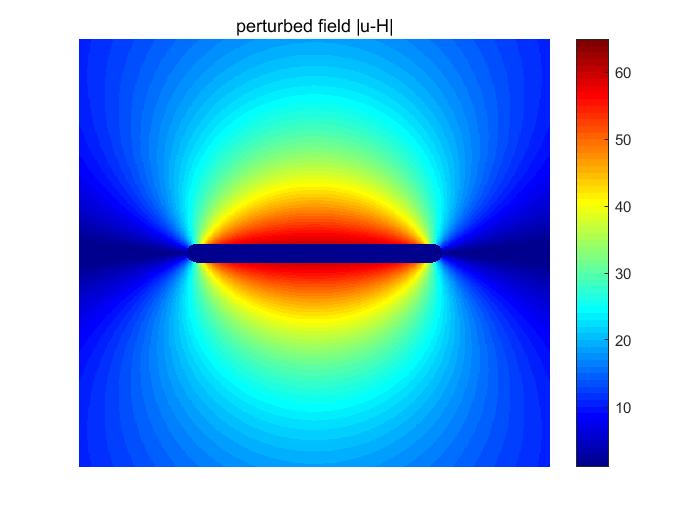}}
{\includegraphics[width=0.4\textwidth]{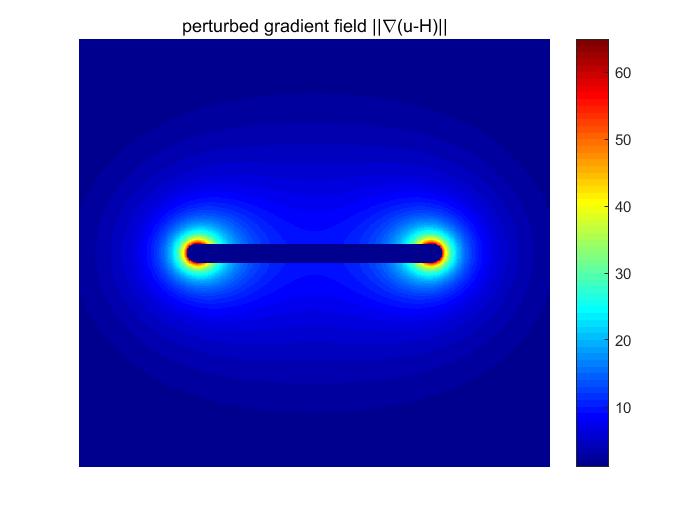}}
\end{center}
\caption{\label{fig:perturbed_02} $\mathbf{a}=(0,1)$. Left: Perturbed field $|u-\mathbf{a}\cdot\Bx|$ (scaled)   Right: Perturbed gradient field $|\nabla u-\mathbf{a}|$ (scaled). }
\end{figure}

\begin{figure}[h]
\begin{center}
{\includegraphics[width=0.4\textwidth]{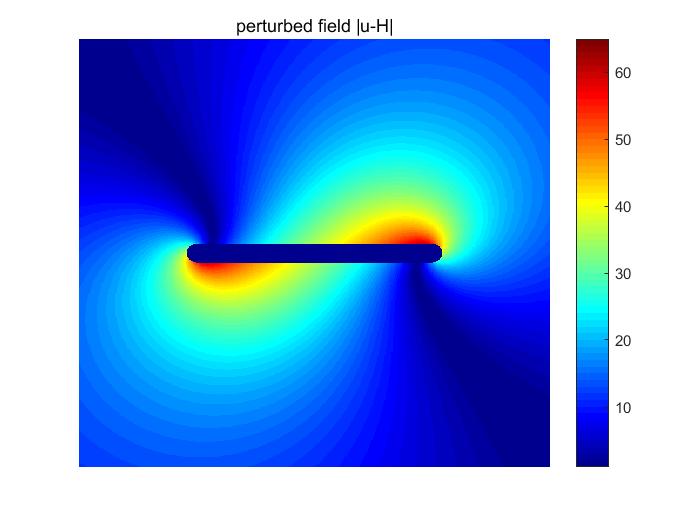}}
{\includegraphics[width=0.4\textwidth]{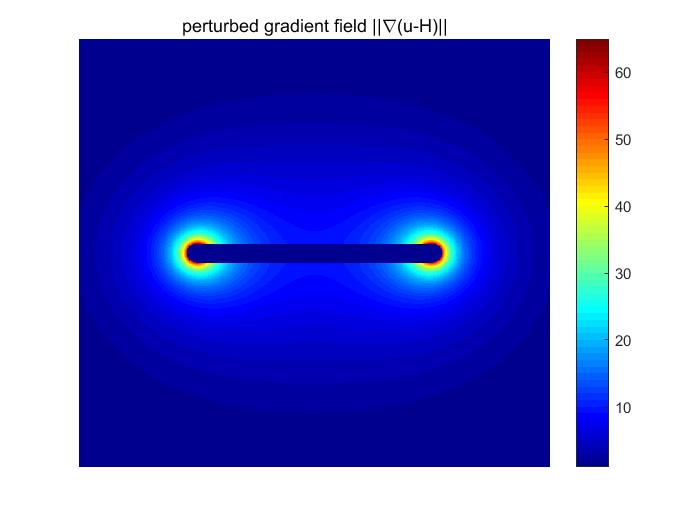}}
\end{center}
\caption{\label{fig:perturbed_03} $\mathbf{a}=(1,1)$. Left: Perturbed field $|u-\mathbf{a}\cdot\Bx|$ (scaled)   Right: Perturbed gradient field $|\nabla u-\mathbf{a}|$ (scaled). }
\end{figure}

\section{Application to Calder\'on inverse inclusion problem}

In this section, we consider the application of the quantitative results derived in the previous section to the Calder\'on inverse inclusion problem. To that end, we let $D$ denote a generic rod inclusion that is obtained through rigid motions performed on special case described in Section~\ref{sect:1.1}. We write $D(L, \delta, \mathbf{z}_0, \sigma_0)$ to signify its dependence on the length $L$, width $\delta$, position $\mathbf{z}_0$ (which is the geometric centre of $D$) as well as the conductivity parameter $\sigma_0$. Consider the conductivity system \eqref{eq:helm01} associated with a generic inclusion described above. The inverse inclusion problem is concerned with recovering the shape of the inclusion, namely $\partial D$, independent of its content $\sigma_0$, by measuring the perturbed electric field $(u-H)$ away from the inclusion. This is one of the central problems in EIT, which forms the fundamental basis for the electric prospecting. The case with a single measurement, namely the use of a single probing field $H$, is a longstanding problem in the literature. The existing results for the single-measurement case are mainly concerned with specific shapes including discs/balls and polygons/polyhedrons \cite{BFS,BFV,CHL,FKS,KS,LT,TT} as well as the other general shapes but with a-priori conditions; see \cite{A,AI,AIP,Bry,Fri,IP}. As discussed earlier, in \cite{LTY}, the local recovery of the highly-curved part of $\partial D$ was also considered. Next, using the asymptotic result quantitative result in Theorem~\ref{thm:3.2}, we shall show that one can uniquely determine a conductive inclusion up to an error level $\delta\ll 1$.

\begin{thm}\label{thm:final}
Let $D_j=D_j(L_j, \delta_j, \mathbf{z}_0^{(j)}, \sigma_0^{(j)})$, $j=1, 2$, be two conductive rods such that $L_j\sim 1, \delta_j\sim \delta\ll 1$ and $\sigma_0^{(j)}\sim 1$ for $j=1, 2$. Let $u_j$ be the corresponding solution to \eqref{eq:helm01} associated with $D_j$ and a given nontrivial $H(\mathbf{x})=\mathbf{a}\cdot \Bx$. Suppose that
\begin{equation}\label{eq:m1}
u_1=u_2\quad\mbox{on}\ \ \partial\Sigma,
\end{equation}
where $\Sigma$ is a bounded simply-connected Lipschitz domain enclosing $D_j$. Then it cannot hold that
\begin{equation}\label{eq:m2}
\mathrm{dist}(D_1, D_2)\gg \delta.
\end{equation}
\end{thm}
\begin{proof}
First, by \eqref{eq:m1}, we know that $u_1=u_2$ in $\mathbb{R}^2\backslash \Sigma$ and hence by unique continuation, we also know that $u_1=u_2$ in $\mathbb{R}^2\backslash (D_1\cup D_2)$. Next, since the Laplacian is invariant under rigid motions, we note that the quantitative result in Theorem~\ref{thm:3.2} still holds for $D_j$. By contradiction, we assume that \eqref{eq:m2} holds. It is easily seen that there must be one cap point, say $\Theta_0\in \partial D_1$, which lies away from $D_2$ and $\mathrm{dist}(\Theta_0, D_2)\gg \delta$. Hence, one has $u_1(\Theta_0)=u_2(\Theta_0)$. Now, we arrive at a contradiction by noting that using Theorem~\ref{thm:3.2}, one has $u_1(\Theta_0)\sim 1$ whereas $u_2(\Theta_0)\sim \delta\ll 1$.

The proof is complete.

\end{proof}

\section*{Acknowledgments}
The work of X. Fang was supported by Humanities and Social Sciences Foundation of the Ministry of Education no. 20YJC910005, Major Project for National Natural Science Foundation of China no. 71991465, PSCF of Hunan No. 18YBQ077 and RFEB of Hunan No. 18B337.
The work of Y. Deng was supported by NSF grant of China No. 11971487 and NSF grant of Hunan No. 2020JJ2038.
The work of H. Liu was supported by a startup fund from City University of Hong Kong and the Hong Kong RGC General Research Funds, 12301218, 12302919 and 12301420.

\end{document}